\documentclass{amsart}
\usepackage{amsfonts}
\usepackage{amsmath}
\usepackage{amssymb}
\usepackage{amsthm}

\newtheorem{theorem}{Theorem}
\newtheorem{cor}{Corollary}
\newtheorem{lemma}{Lemma}
\newtheorem{exmp}{Example}
\newtheorem{rem}{Remark}

\begin{document}
\author{Mark Pankov}
\title{Automorphisms of infinite-dimensional hypercube graph}
\address{Department of Mathematics and Informatics, University of Warmia and Mazury,
{\. Z}olnierska 14A, 10-561 Olsztyn, Poland}
\email{pankov@matman.uwm.edu.pl markpankov@gmail.com}

\maketitle
\begin{abstract}
We consider the infinite-dimensional hypercube graph.
This graph is not connected and has isomorphic connected components. 
We desribe the restrictions of its automorphisms
to the connected components and the automorphism group of connected component.
\end{abstract}

\section{Introduction}
By \cite{ER}, typical graphs have no non-trivial automorphisms.
On the other hand, the classical Frucht  result \cite{Frucht} states that every abstract group can be realized as
the automorphism group of some graph (we refer \cite{Cam} for more information concerning graph automorphisms).
In particular, the Coxeter group of type ${\textsf B}_{n}={\textsf C}_{n}$ (the wreath product $S_{2}\wr S_{n}$) is isomorphic
to the automorphism group of the $n$-dimensional hypercube graph $H_{n}$.

In this note we consider the infinite-dimensional hypercube graph $H_{{\aleph}_{0}}$.
This graph is not connected and has isomorphic connected components.
We describe the restrictions of its automorphisms
to the connected components (Theorem 1). As a simple consequence, we establish that
the automorphism group of connected component is isomorphic to
the so-called weak wreath product of $S_{2}$ and $S_{\aleph_{0}}$ (Corollary 1).

\section{Infinite-dimensional hypercube graph}

A subset $X\subset {\mathbb Z}\setminus\{0\}$ is said to be {\it singular} if
$$i\in X\;\Longrightarrow\;-i\not\in X.$$
For every natural $i$ each maximal singular subset contains precisely one of the numbers $i$ or $-i$;
in other words, if $X$ is a maximal singular subset then the same holds for its complement in ${\mathbb Z}\setminus\{0\}$.
Two maximal singular subsets $X,Y$ are called {\it adjacent} if
$$|X\setminus Y|=|Y\setminus X|=1.$$
In this case, we have
$$X=(X\cap Y)\cup\{i\}\;\mbox{ and }\;Y=(X\cap Y)\cup\{-i\}$$
for some number $i\in {\mathbb Z}\setminus\{0\}$.

Following Example 2.6 in \cite{Pankov-book}, we say that a permutation $s$ on the set ${\mathbb Z}\setminus \{0\}$
is {\it symplectic} if
$$s(-i)=-s(i).$$
A permutation is symplectic if and only if it preserves
the family of singular subsets.
The group of symplectic permutations is isomorphic to the wreath product $S_{2}\wr S_{\aleph_{0}}$
(we write $S_{\alpha}$ for the group of permutations on a set of cardinality $\alpha$, 
see Section 5 for the definition of wreath product).
The action of this group on the family of maximal singular subsets is transitive.

Denote by $H_{{\aleph}_{0}}$ the graph whose vertex set is formed by all maximal singular subsets and whose edges are
adjacent pairs of such subsets.
This graph is not connected. The connected component containing $X\in H_{{\aleph}_{0}}$ will be denoted by $H(X)$;
it consists of all $Y\in H_{{\aleph}_{0}}$ such that
$$|X\setminus Y|=|Y\setminus X|<\infty.$$
Any two connected components $H(X)$ and $H(Y)$ are isomorphic.
Indeed, every symplectic permutation $s$ on the set ${\mathbb Z}\setminus \{0\}$
induces an automorphism of $H_{{\aleph}_{0}}$; this automorphism transfers $H(X)$ to $H(Y)$
if $s(X)=Y$.

\begin{rem}{\rm
It is clear that $H_{{\aleph}_{0}}$ can be identified with the graph whose vertices are sequences
$$\{a_{n}\}_{n\in {\mathbb N}}\;\mbox{ with }\;a_{n}\in \{0,1\}$$
and $\{a_{n}\}_{n\in {\mathbb N}}$ is adjacent with $\{b_{n}\}_{n\in {\mathbb N}}$ (connected by an edge) if
$$\sum_{n\in {\mathbb N}}|a_{n}-b_{n}|=1.$$
Then one of the connected components is formed by all sequences having a finite number of non-zero elements.
The graph $H_{{\aleph}_{0}}$ also can be defined as the Cartesian product of infinitely many
factors $K_2$ \cite{ImRall}.
}\end{rem}

\section{Automorphisms}

Every automorphism of $H_{{\aleph}_{0}}$ induced by a  symplectic permutation will be called {\it regular}.
An easy verification  shows that distinct symplectic permutations induce distinct regular automorphisms.
Therefore, the group of regular automorphisms is isomorphic to $S_{2}\wr S_{\aleph_{0}}$.

Non-regular automorphisms exist.
The following example is a modification of examples given in \cite{BH,Pankov1}, see also Example 3.14 in \cite{Pankov-book}.

\begin{exmp}{\rm
Let $A\in H_{{\aleph}_{0}}$ and $B$ be a vertex of the connected component $H(A)$ distinct from $A$.
We take any symplectic permutation $s$ transferring  $A$ to $B$.
This permutation preserves $H(A)$ and
the mapping
$$f(X):=
\begin{cases}
s(X)&X\in H(A)\\
\;X&X\in H_{{\aleph}_{0}}\setminus H(A)
\end{cases}$$
is well-defined.
Clearly, $f$ is a non-trivial automorphism of $H_{{\aleph}_{0}}$.
Suppose that this automorphism is regular and $t$ is the associated symplectic permutation.
For every $i\in {\mathbb Z}\setminus \{0\}$ there exists a singular subset $N$ such that
$$X=N\cup\{i\}\;\mbox{ and }\;Y=N\cup\{-i\}$$
are elements of $H_{{\aleph}_{0}}\setminus H(A)$.
Then
$$t(N)=t(X\cap Y)=t(X)\cap t(Y)=f(X)\cap f(Y)=X\cap Y=N$$
and
$$N\cup\{i\}=X=f(X)=t(X)=t(N)\cup\{t(i)\}=N\cup\{t(i)\}$$
which implies that $t(i)=i$.
Thus $t$ is identity which is impossible.
So, the automorphism $f$ is non-regular.
}\end{exmp}

\begin{theorem}
The restriction of every automorphism of $H_{{\aleph}_{0}}$ to any connected component
coincides with the restriction of some regular automorphism to this connected component.
\end{theorem}

\begin{rem}{\rm
A similar result was obtained in \cite{Pankov1} for the infinite Johnson graph. The proof
of that result  is based on the same idea, but technically is more complicated.
}\end{rem}

\section{Proof of Theorem 1}
Let $A\in H_{{\aleph}_{0}}$ and $f$ be the restriction of an automorphism of $H_{{\aleph}_{0}}$ to
the connected component $H(A)$.
For every $X\in H_{{\aleph}_{0}}$ we denote by $X^{\sim}$ the set which contains $X$ and all vertices of $H_{{\aleph}_{0}}$
adjacent with $X$.
It is clear that $X^{\sim}$ is contained in $H(A)$ if $X\in H(A)$.

\begin{lemma}
For every $X\in H(A)$ there is a symplectic permutation $s_{X}$
such that
\begin{equation}\label{eq1}
f(Y)=s_{X}(Y)\;\;\;\;\;\;\forall\;Y\in X^{\sim}.
\end{equation}
\end{lemma}

\begin{proof}
We can assume that $f(X)$ coincides with $X$
(if $f(X)\ne X$ then we take any symplectic permutation $t$ sending $f(X)$ to $X$ and consider $tf$).
In this case, the restriction of $f$ to $X^{\sim}$ is a bijective transformation of $X^{\sim}$.

For every $i\in {\mathbb Z}\setminus \{0\}$ one of the following possibilities is realized:
\begin{enumerate}
\item[$\bullet$] $i\not\in X$,
\item[$\bullet$] $i\in X$.
\end{enumerate}
Consider the first case. Then $-i\in X$ and there is unique element of $X^{\sim}$ containing $i$, this is
\begin{equation}\label{eq2}
Y=\{i\}\cup(X\setminus \{-i\}).
\end{equation}
Since $f|_{X^{\sim}}$ is a transformation of $X^{\sim}$,
$f(Y)$ is adjacent with $X$ and the set $f(Y)\setminus X$ contains only one element. We denote it by $s_{X}(i)$.
It is clear that $s_{X}(i)\not\in X$.

In the second case,  $-i\not\in X$ and we define $s_{X}(i)$ as $-s_{X}(-i)$.
Since $s_{X}(-i)$ does not belong to $X$, we have $s_{X}(i)\in X$.

So, $s_{X}$ is a symplectic permutation on ${\mathbb Z}\setminus \{0\}$ such that
$$s_{X}(X)=X.$$
Now, we check \eqref{eq1}.

Let $Y\in X^{\sim}$. Then we have \eqref{eq2}
for some $i$ and
$$s_{X}(Y)=\{s_{X}(i)\}\cup (s_{X}(X)\setminus \{-s_{X}(i)\})=\{s_{X}(i)\}\cup (X\setminus \{-s_{X}(i)\})$$
is unique element of $X^{\sim}$ containing $s_{X}(i)$. On the other hand,
$s_{X}(i)$ belongs to $f(Y)$ by the definition of $s_{X}$.
Therefore, $f(Y)$ coincides with $s_{X}(Y)$.
\end{proof}

\begin{lemma}
If $X,Y\in H(A)$ are adjacent then $s_{X}=s_{Y}$.
\end{lemma}

\begin{proof}
Since $X,Y$ are adjacent, we have
$$X=\{i\}\cup(X\cap Y)\;\mbox{ and }\;Y=\{-i\}\cup(X\cap Y)$$
for some $i\in X$.
We can assume that
$$f(X)=X\;\mbox{ and }\;f(Y)=Y.$$
Indeed, in the general case
$$f(X)=\{j\}\cup(f(X)\cap f(Y))\;\mbox{ and }\;f(Y)=\{-j\}\cup(f(X)\cap f(Y))$$
(since $f(X)$ and $f(Y)$ are adjacent);
we take any symplectic permutation $t$ sending $j$ and $f(X)\cap f(Y)$ to $i$ and $X\cap Y$ (respectively)
and consider $tf$.

Then
$$s_{X}(X\cap Y)=s_{X}(X)\cap s_{X}(Y)=f(X)\cap f(Y)=X\cap Y;$$
similarly,
$$s_{Y}(X\cap Y)=X\cap Y.$$
We have
$$(X\cap Y)\cup\{i\}=X=f(X)=s_{X}(X)=s_{X}((X\cap Y)\cup\{i\})=(X\cap Y)\cup\{s_{X}(i)\}$$
and the same arguments show that
$$(X\cap Y)\cup\{i\}=(X\cap Y)\cup\{s_{Y}(i)\}.$$
Therefore,
$$s_{X}(i)=s_{Y}(i)=i\;\mbox{ and }\;s_{X}(-i)=s_{Y}(-i)=-i.$$
Now, we show that the equality
\begin{equation}\label{eq3}
s_{X}(j)=s_{Y}(j)
\end{equation}
holds for every $j\ne \pm i$. Since $s_{X}$ and $s_{Y}$ are symplectic, it is sufficient to establish \eqref{eq3} only in the case when
$j\not\in X\cup Y$. Indeed, if $j\in X\cap Y$ then $-j$ does not belong to $X\cup Y$.

Let $j$ be an element of ${\mathbb Z}\setminus \{0\}$ which does not belong to $X\cup Y$.
Then $-j\in X\cap Y$ and
$$X':=\{j\}\cup (X\setminus \{-j\})\in X^{\sim},\;\;Y':=\{j\}\cup (Y\setminus \{-j\})\in Y^{\sim}$$
are adjacent. Hence
$$f(X')=s_{X}(X')=\{s_{X}(j)\}\cup (X\setminus \{-s_{X}(j)\})$$
and
$$f(Y')=s_{Y}(Y')=\{s_{Y}(j)\}\cup (Y\setminus \{-s_{Y}(j)\})$$
are adjacent. The latter is possible only in the case when $s_{X}(j)=s_{Y}(j)$.
\end{proof}

Using the connectedness of $H(A)$ and Lemma 2, we establish that $s_{X}=s_{Y}$ for all $X,Y\in H(A)$.

\section{Automorphisms of connected components}
Let $G_{1}$ and $G_2$ be permutation groups on sets $X_1$ and $X_2$, respectively.
Recall that the {\it wreath product} $G_1\wr G_2$ is a permutation group on $X_1\times X_2$ and its elements are
compositions of the following two types of permutations:
\begin{enumerate}
\item[(1)] for each element $g\in G_{2}$, the permutation $(x_1,x_{2})\to (x_1, g(x_{2}))$;
\item[(2)] for each function $i: X_{2}\to G_1$, the permutation $(x_1,x_{2})\to (i(x_{2})x_1,x_{2})$.
\end{enumerate}
Consider the subgroup of $G_1\wr G_2$ whose elements are compositions of all permutations of type (1) and permutations
of type (2) such that the set
$$\{\;x_{2}\in X_2\;:\;i(x_{2})\ne {\rm id}_{X_{1}\;}\}$$
is finite. This is a proper subgroup only in the case when $X_2$ is infinite;
it is will be called the {\it weak wreath product} and denoted by $G_1\wr_{w} G_2$.

\begin{cor}
The automorphism group of connected component of $H_{{\aleph}_{0}}$ is isomorphic
to the weak wreath product $S_{2}\wr_{w} S_{{\aleph}_{0}}$.
\end{cor}

\begin{proof}
Let $A\in H_{{\aleph}_{0}}$ and $f$ be an automorphism of the connected component $H(A)$.
By the previous section, $f$ is induced by a symplectic permutation $s$.
Since $f(A)=s(A)$ belongs to $H(A)$,
the set $s(A)\setminus A$ is finite.
So, the automorphism group of $H(A)$ is isomorphic to the group of symplectic permutations $s$
such that the set $s(A)\setminus A$ is finite.
The latter group is isomorphic to the weak wreath product $S_{2}\wr_{w} S_{{\aleph}_{0}}$
(indeed, we  can identify the set ${\mathbb Z}\setminus\{0\}$ with the Cartesian product ${\mathbb Z}_{2}\times A$
and the group $S_{{\aleph}_{0}}$ with the group of all permutation on $A$).
\end{proof}

\begin{rem}{\rm
The latter result is not new.
Since $H_{{\aleph}_{0}}$ is the Cartesian product of infinitely many
factors $K_2$, Corollary 1 can be drawn from the well-known results concerning the automorphism group of 
Cartesian product of graphs \cite{Im,Miller}.
}\end{rem}

\subsection*{Acknowledgement}
I express my deep gratitude to Wilfried Imrich for useful information.

\end{document}